\documentclass[10pt,twoside]{article}
\usepackage{amsmath,amssymb,amsthm}
\usepackage{graphicx}
\usepackage{amsmath, amsthm, amsfonts, amssymb, multicol, pstricks}
\usepackage{color}

\newtheorem{theorem}{Theorem}[section]
\newtheorem{corollary}[theorem]{Corollary}
\newtheorem{lemma}[theorem]{Lemma}
\newtheorem{prop}[theorem]{Proposition}
\theoremstyle{definition}
\newtheorem{definition}[theorem]{Definition}
\newtheorem{example}[theorem]{Example}
\newtheorem{remark}[theorem]{Remark}
\numberwithin{equation}{section}

\newcommand{\al}{\alpha}

\newcommand{\om}{\omega}

\newcommand{\ve}{\varepsilon}

\newcommand{\vp}{\varphi}
\newcommand{\real}{\mathbb{R}}
\newcommand{\rn}{\real^{n}}

\numberwithin{equation}{section}

\allowdisplaybreaks
\usepackage{color}
\begin{document}

\title{Hybrid measure differential equations: existence of solutions and continuous dependence on parameters.}

\author{Claudio A. Gallegos\thanks{ Universidad de Chile, UCH, Departamento de Matem\'{a}ticas, Casilla 653, Santiago, CHILE. E-mail: {\tt claudio.gallegos.castro@gmail.com}}, {Hern\'an R.} Henr\'{i}quez\thanks{Universidad de Santiago de Chile, USACH, Departamento de Matem\' atica, Casilla 307, Correo 2, Santiago, CHILE. E-mail: {\tt hernan.henriquez@usach.cl}},
		and Jaqueline G. Mesquita\thanks{ Universidade de Bras\'{i}lia, Departamento de Matem\'{a}tica, Campus Universit\'{a}rio Darcy Ribeiro, Asa Norte 70910-900, Bras\'{i}lia-DF, BRAZIL. E-mail: {\tt jgmesquita@unb.br}}
}

\date{}
\maketitle

\begin{abstract}
This paper is devoted to study the qualitative properties of  hybrid measure differential
 equations  (HMDEs, for short).
We establish several results on the existence of global solutions, including the existence of regulated,
continuous, differentiable and S--asymptotically $\om$--periodic solutions. Furthermore, we present a result on continuous dependence
of solutions in terms of parameters. Our results are based on  extensions of Kasnoselskii's
fixed-point theorem.


\end{abstract}

\pagestyle{myheadings} \markboth{\hfil C. A. Gallegos,  H. R. Henr\'{\i}quez and
J. G. Mesquita \hfil $\hspace{3cm}$ } {\hfil$\hspace{1.5cm}$
{Hybrid measure differential equations}
\hfil}

\section{Introduction}
This work is dedicated to the study of qualitative properties of the hybrid measure differential equation (abbreviated, HMDE) {given by:}
\begin{eqnarray}
x(t) &=& x_{0} - h(t_{0}, x(t_0)) + h(t,x(t)) + \int_{t_0}^{t}f(s,x(s)){\rm d}g(s), \; \; t\in J,  \label{hmde1} \\
x(t_0)& = & x_0\in\rn, \label{hmde2}
\end{eqnarray}
where $J=[t_{0},t_{0}+a]$ with $t_{0} \in \real$ and  $a>0$. In this problem, $f, h : J \times \rn \to \rn$ and
$g:J\to\real$ denote functions whose properties will be described later, and the integral on the right hand side is considered in the sense of Kurzweil--Stieltjes.

Specifically, our main concern in this work is the study of existence of solutions for problem HMDE \eqref{hmde1}--\eqref{hmde2}. We {analyze} this equation by imposing certain technical conditions over the function $f$, most of which are related to the concept of integration in the sense of Kurzweil-Stieltjes, see conditions (A1)-(A2), and (A2*) in Section~3. Further, we consider the function $h$ satisfying a nonlinear contraction assump\-tion, see conditions (A3) and (A3*) in Section~3. The above mentioned conditions allow us to guarantee the existence of solutions for problem HMDE  \eqref{hmde1}--\eqref{hmde2}, see Theorems~\ref{exist1} and \ref{T1}. Another goal in this paper is to establish special classes of solutions for problem HMDE \eqref{hmde1}--\eqref{hmde2}, including continuous and differentiable solutions, see Propositions~\ref{C2} and \ref{C3} in Section~4. Additionally, we study solutions defined on unbounded intervals in order to  es\-ta\-blish S--Asymptotically $\om$-periodic solutions for problem HMDE \eqref{hmde1}--\eqref{hmde2}, see Theorems~\ref{thm3.2} and \ref{thm3.3}. We emphasize that this concept of periodicity have yet to be considered in the framework of measure differential equations. Besides of the aforementioned existence results, we establish a theorem of continuous dependence on parameters for problem HMDE \eqref{hmde1}--\eqref{hmde2}, see Theorem~5.1 in Section 5.

Some of the earliest contributors in the theory of measure differential equations were W. W. Schmae\-deke, R. R. Sharma  and P. C. Das (see \cite{DS,Sch,Sharma}). Over the years, several authors have been developed and motivated the qualitative theory of these equations, principally focused on the existence of solutions, stability theory, and applications to other types of systems exhibiting discontinuous behaviour, see for instance \cite{CS,DMS1,DMS2,FGMT1,M-J-A,GHM,LM,MS,satco}. One of the interests in studying this type of equations relies on its formulation, which can include two types of problems widely studied in the specialized literature nowadays, such as is the case of dynamic equations on time scales and impulsive equations. 

In relation to dynamic equations, the notion of time scales has been intensively developed in dynamical systems and related topics. Some of the reasons are the capacity of this theory  to unify the discrete and continuous analysis, include some other developments such as quantum analysis, and its applications to the modeling of strongly nonlinear systems. We refer to \cite{ABL,Bohner1, Bohner2,HM}. Recently, an interesting theoretical point of view it was considered in the time scales setting.
In \cite{AS}, A. Slav\'ik gave a first connection {between} measure differential equations and dynamic equations on time scales. Henceforth, this innovating approach has been useful to develop the theory of dynamic equations on time scales from a general perspective, allowing to consider more general conditions than the usual ones. We refer to
\cite{FGMT1,M-J-A,M-J-A2,FMS,GGM,GHM} for recent advances.

Let us consider a time scale $\mathbb{T}$, {\it i.e.}, a nonempty closed subset of $\mathbb{R}$, and let $\Delta$ be the operator delta derivative.
The analogue of the problem HMDE \eqref{hmde1}-\eqref{hmde2}  in the framework of time scales can be formulated by
\begin{equation}\label{DETS}
[x(t)-h(t,x(t))]^{\Delta}=f(t,x(t)),\; \; t\in J\cap\mathbb{T},
\end{equation}
{for which, under appropriate assumptions (see \cite{ZSLB}), its integral form is given by
\begin{equation}\label{intDETS}
x(t)=x_0-h(t_0,x(t_0))+h(t,x(t))+\int_{t_0}^{t}f(t,x(t))\Delta s,\; \; t\in J\cap\mathbb{T}.
\end{equation}}

As previously mentioned, due to the development carried out in \cite{AS}, we can insert the equation \eqref{intDETS} in the context of
measure differential equations of type \eqref{hmde1}.
In consequence, the qualitative theory developed in this paper  can be applied to  dynamic equations of type  \eqref{intDETS}.
For instance, Theorems~\ref{exist1} and \ref{T1} improve the existence results founded in \cite{ZSLB}, and also it is possible to provide
new existence results of S-asymptotically $\om$-periodic solutions for \eqref{intDETS}.

Furthermore, it is also possible to include as a special case of measure differential equations the so-called
\emph{impulsive differential equations}. Indeed, it is a known fact that measure differential equations of type \eqref{hmde1} encompass
the following hybrid impulsive system
\begin{align}
[x(t) - h(t, x(t))]^{\prime} &= f(t, x(t)), \ \ t \neq \tau_{j}, \label{impulsive-1}  \\
\Delta^+ x(\tau_j) &= I_j (x(\tau_j)), \quad j \in \{1, \ldots, m\}, \label{impulsive-2}
\end{align}
where $f,h \colon J \times\mathbb R^{n} \to \mathbb R^{n}$, $I_{j} \colon \mathbb R^n \to \mathbb R^n$ for each $j\in \{1, \ldots, m\}$,
$\{\tau_j\}_{j\in \{1, \ldots, m\}}$ is an increasing real sequence, and $ \Delta^{+} x(\tau_j)=x(\tau_j^+) - x(\tau_j)$.

In the classical theory of ordinary differential equations, the equation \eqref{impulsive-1} is called hybrid differential equation with linear perturbation of second type. We refer the reader to
\cite{DHAGE2,LSYT,ZSLB} and references therein for details on this topic.

Under appropriate assumptions, we can rewrite \eqref{impulsive-1}-\eqref{impulsive-2} in the following integral form
\begin{equation}\label{impulsive-integral-form}
x(t) =x(t_0)+ h(t, x(t)) - h(t_0, x(t_0)) +  \int_{t_0}^{t} f(s, x(s))\,{\rm d}s + \!\!\! \sum_{\substack{j \in \{1, \ldots, m\}, \\
t_{0} \leqslant \tau_{j} <t}}I_{j} (x(\tau_j)),
\end{equation}
where the integral on the right--hand side is considered in the Kurzweil sense. Note that the equation \eqref{impulsive-integral-form}
is a special case of the measure differential equation \eqref{hmde1}. For more details about this aspect, we mention \cite{M-J-A,MST}.
Moreover, using the correspondence found in \cite{M-J-A}, it is possible to regard measure differential equations of type \eqref{hmde1}
to study the hybrid impulsive dynamic equation on time scales given by
\begin{equation}\label{impulsive-integral-form-1}
x(t) =x(t_0)+ h(t, x(t)) - h(t_0, x(t_0)) +  \int_{t_0}^{t} f(s, x(s)) \Delta s + \!\!\! \sum_{\substack{j \in \{1, \ldots, m\}, \\
t_{0} \leqslant \tau_{j} < t}}I_{j} (x(\tau_j)),
\end{equation}
where the integral on the right hand side of \eqref{impulsive-integral-form-1} represents the $\Delta$-integral on the time scale {in the sense of Kurzweil}.
Consequently, the qualitative theory
obtained in this paper  can be also translated to these hybrid impulsive cases. The equations given by \eqref{impulsive-integral-form} and \eqref{impulsive-integral-form-1}  have many important applications, since they can describe models which have abrupt state changes.
This type of behavior can be found in many different situations and phenomena such as bursting rhythm models in medicine, frequency
modulated systems, ingestion of medicine, among others. We refer to  \cite{BHNO,BHNO2,Braverman-2,Braverman,Li-Bohner} and the references therein.

This paper is organized in five sections. In Section 2, a brief summary of the Kurzweil--Stieltjes integral and properties
that we will use later are included. In Section 3, we provide existence results for problem HMDE \eqref{hmde1}--\eqref{hmde2}. In Section 4, we investigate the existence of solutions with special properties as continuity, differentiability and the existence of
S--asymptotically $\omega$--periodic solutions. In Section 5, we provide a continuous dependence on parameters result for problem HMDE \eqref{hmde1}-\eqref{hmde2}.

\section{Preliminaries}
In this section, we present some basic concepts of Kurzweil--Stieltjes integration theory. For more details, the reader can see
\cite{MST,SCHWABIK1}. {Throughout this text, $X$ will always denote a Banach space endowed with a norm $\|\cdot\|$ and $[a,b]\subset\real$ is a compact interval}. In particular, we consider $\mathbb{R}^{n}$ endowed with a norm $\|\cdot\|$.

A function $\delta \colon [a,b]\to\mathbb{R}^{+}$ is called a {\itshape gauge} on $[a,b]$. If $\delta$ is a gauge on $[a, b]$, a
 \emph{tagged partition} of the interval $[a,b]$ with subdivision points
$a=s_0\leqslant s_1\leqslant \cdots \leqslant s_k=b$,  and {\itshape tags} $\tau_i\in [s_{i-1},s_i]$, $i=1,...,k$, is called  $\delta$-{\itshape fine} if
\[
[s_{i-1},s_i] \subset \left(\tau_i-\delta(\tau_i), \tau_i+\delta(\tau_i)\right), \; \; i=1, \ldots, k.
\]
It is important to mention that given a gauge $\delta \colon [a, b] \to\mathbb{R}^{+}$, it is always possible to obtain a $\delta$-fine tagged partition of the interval $[a, b]$, see \cite[Cousin's lemma~6.1.3]{MST}.

\begin{definition}\label{KS-int} {(\cite[Definition 6.1.2]{MST})}
We say that a function $f \colon [a, b] \to \mathbb R^n$ is {\em Kurzweil--Stieltjes integrable} on $[a, b]$ with respect to a function
$g\colon [a, b] \to \mathbb R$ if there is a vector $\mathcal{I}\in\mathbb R^n$ such that for every $\varepsilon>0,$ there exists a gauge $\delta \colon [a, b]\rightarrow \mathbb{R}^{+}$  such that
\[
\left\|\sum\limits_{i=1}^{k}f(\tau_i)\left(g(s_i)-g(s_{i-1})\right) - \mathcal{I}\right\|<\varepsilon
\]
for all $\delta$--fine tagged partition of $[a, b].$ In this case, $\mathcal{I}$ is called the {\em Kurzweil--Stieltjes integral of $f$ with respect to $g$ over $[a, b]$} and it will be denoted by $\int^{b}_{a} f(s)  {\rm d} g(s),$ or just by $\int^{b}_{a} f {\rm d} g$.
\end{definition}

As it should be expected, the Kurzweil--Stieltjes integral satisfies the usual properties of linearity, integrability on subintervals, additivity with respect to adjacent intervals,  among others  usual properties in integration theory (see \cite{MST}).

When $g$ is the identity function, R. Henstock in 1961 (see \cite{H1}) introduced another concept of integral for real valued functions
which is equivalent to the integral in the sense of Kurzweil. In the technical literature, due to its equivalence with the integral
of Kurzweil, this integral is known as \emph{Henstock--Kurzweil integral} or \emph{gauge integral} (see \cite{KS1}).
The relationship between different integration concepts can be summarized in the following chain of inclusions
\[
\mathcal{R}([a,b],\mathbb{R})\subset\mathcal{L}_1([a,b],\mathbb{R})\subset H([a,b],\mathbb{R})=\mathcal{K}([a,b],\mathbb{R}),
\]
where $\mathcal{R}([a,b],\mathbb{R})$ denotes the space of Riemann integrable functions, $\mathcal{L}_1([a,b],\mathbb{R})$ stands for
the space of Lebesgue integrable functions, $H([a,b],\mathbb{R})$ denotes the space of Henstock integrable functions,
and $\mathcal{K}([a,b],\mathbb{R})$ denotes the space of Kurzweil integrable functions.

To complete these general comments on the Kurz\-weil-\-Stielt\-jes integral,
we point out that the Kurzweil--Stieltjes definition allows us to integrate a broad class of functions.
For instance, as it was explained previously, if $g(t)\equiv t$, then we obtain the Henstock--Kurzweil integral,
allowing us to integrate highly oscillating functions.
On the other hand, the Kurzweil--Stieltjes  integral allows us to integrate a function $f$ with respect to $g$ even when
both of them have points of discontinuity in common, which we know that is not possible for the integral in the Riemann--Stieltjes  sense.
Therefore, for all these reasons, it seems most convenient to work with the Kurzweil--Stieltjes integral when continuity of both
functions $f$ and $g$ is not required. The relation among other types of Stieltjes integrals and the
Kurzweil--Stieltjes can be found in \cite{MST}.

In the rest of this section, we develop in detail some fundamental technical aspects to justify the results presented in
the following sections. Initially, we recall the concept of regulated function  which plays an important role in this work.

\begin{definition}
A function $f \colon [a,b]\to X$ is called {\itshape regulated} if the limits below exist
\[
\lim_{s\to t^-}f(s)=f(t^{-}) \ \ \textrm{for} \ \ t\in(a,b] \ \ \textrm{and} \ \ \lim_{s\to t^{+}}f(s)=f(t^{+}) \ \ \textrm{for} \ \ t\in[a,b).
\]
\end{definition}
The space consisting of all regulated functions $f \colon [a,b]\to X$ will be denoted by $G([a,b],X)$, and it is a Banach space
under the usual norm of the uniform convergence $\|f\|_{\infty}=\displaystyle\sup_{a \leqslant t \leqslant b} \|f(t)\|$.

In addition, we recall that the vector space consisting of all functions $f\colon [a,b]\to X$ of bounded variation on
$[a,b]$, denoted by $BV([a,b],X)$, is a Banach space endowed with the norm
\[
\|f\|_{BV}= \|f(a)\|+ \mathrm{var}_{a}^{b}(f), \; f \in BV([a,b],X),
\]
where $\mathrm{var}_{a}^{b}(f)$ denotes the variation of the function $f$ on $[a,b]$. In the case $X=\real$, we
abbreviate the notation by writing $BV([a,b])$ instead of $BV([a,b],\real)$. It is well known  that $BV([a,b],X)\subset G([a,b],X)$.

The next result yields  sufficient conditions to ensure the existence of the Kurzweil--Stieltjes integral.
\begin{theorem} \label{thm2.1} {\rm (\cite[Corollary~1.34]{SCHWABIK1})}
Let $f \colon [a,b]\to \mathbb{R}^{n}$ be a regulated function and {let} $g \colon [a,b]\to\mathbb{R}$ be a bounded variation function.
Then the integral $\int_{a}^{b} f {\rm d} g$ exists, and
\[
\left\|\int_{a}^{b}f(s)  {\rm d} g(s) \right\| \leqslant \|f\|_{\infty}\mathrm{var}_{a}^{b}(g).
\]
\end{theorem}

The following results establish important properties of the Kurzweil--Stieltjes integral.
\begin{theorem}  \label{thm1.2} {\rm (\cite[Corollary 6.5.4]{MST})}
Let $f \colon [a,b]\to\mathbb{R}^n$ and $g \colon [a,b]\to\mathbb{R}$ be a pair of functions such that $g$ is regulated, and the
integral $\int_{a}^{b} f  {\rm d} g$ exists. Then the function  $p \colon [a,b] \to \mathbb{R}^{n}$ given by
\[
p(t)=\int_{a}^{t} f(s) {\rm d} g(s), \;\; t \in [a,b],
\]
is regulated, and satisfies
\begin{align*}
p(t^+)&=p(t)+f(t)\Delta^{+}g(t), \; \; t\in [a,b), \\
p(t^-)&=p(t)-f(t)\Delta^{-}g(t), \; \; t\in (a,b],
\end{align*}
where $\Delta^{+}g(t)=g(t^+)-g(t)$ and $\Delta^{-}g(t)=g(t)-g(t^-)$.
\end{theorem}
\begin{remark}
In the case that $g\colon [a,b]\to\real$ is a regulated function, the Kurzweil--Stieltjes integral $\int_{a}^{b}f{\rm d}g$ exists
for $f\in BV([a,b],\rn)$, see \cite[Theorem~6.3.11]{MST}.
\end{remark}

The next definition will allow us to appropriately manage sets of regulated functions.
\begin{definition} \label{D3} (\cite{FRANKOVA})
A set $\mathcal{F}\subset G([a,b],\rn)$ is said to be {\it equiregulated} if for every $\varepsilon>0$
the following conditions hold:
\begin{itemize}
\item [(i)] For every $\tau_{0} \in (a,b]$, there exists $\delta>0$  such that
 $\|x(t)-x(\tau_{0}^{-})\|<\varepsilon$ for all  $\tau_{0} - \delta < t <\tau_{0}$ and all $x\in\mathcal{F}$.
\item [(ii)] For every $\tau_{0} \in [a,b)$, there exists $\delta>0$  such that
$\|x(s) - x(\tau_{0}^{+})\|<\varepsilon$ for all  $\tau_{0} < s <\tau_{0} + \delta$ and all $x\in\mathcal{F}$.
\end{itemize}
\end{definition}

Next we present a type of Arzel\`a--Ascoli Theorem to characterize the compact sets in the space of regulated functions.

\begin{lemma} \label{L1} {\rm (\cite[Corollary~2.4]{FRANKOVA})}
A set $\mathcal{F} \subset G([a,b],\rn)$ is relatively compact if and only if
for every $t \in [a,b]$ the set $\{x(t): x\in \mathcal{F} \}$ is bounded in $\rn$
and $\mathcal{F}$  is equiregulated.
\end{lemma}

We finish this section with a generalization of Lebesgue's Dominated Convergence Theorem.
\begin{theorem}\label{DCT}{\rm (\cite[Theorem~6.8.10]{MST})}
Let $g\in BV([a, b])$ and {let} $f, f_k : [a, b] \to \rn$, $k\in\mathbb{N}$, be functions such that the integral
$\int_{a}^{b}f_k{\rm d}g$
exists for every $k\in\mathbb{N}$ and $\displaystyle\lim_{k\to+\infty}f_k(t)=f(t)$ for each $t\in[a,b]$. Assume that there exists a
positive constant $K>0$ such that for every subdivision $a=\sigma_0<\sigma_1< \ldots <\sigma_{l} =b$ of the interval $[a, b]$ and
every finite subset $\{m_1, m_2,..., m_l\}$ of $\mathbb{N}$, the inequality
\[
\left\|\sum_{j=1}^{l}\int_{\sigma_{j-1}}^{\sigma_{j}} f_{m_{j}}(s)  {\rm d} g(s) \right\| < K
\]
holds. Then the integral $\int_{a}^{b} f  {\rm d} g$ exists, and
\[
\lim_{k\to+\infty}\int_{a}^{b}f_k{\rm d}g=\int_{a}^{b}f{\rm d}g.
\]
\end{theorem}

\section{Existence results}
In this section, we present some results on the existence of solutions for problem HMDE \eqref{hmde1}-\eqref{hmde2}. In order to state our results, we will make use of some additional concepts and a variant of Krasnoselskii's fixed-point theorem.
\begin{definition}
A continuous nondecreasing function $\vp:\real^{+}\to\real^{+}$ is said to be a \emph{$\mathfrak{D}$--function} if $\vp(0)=0$
and $\vp(s)>0$ for $s>0$.
\end{definition}
\begin{definition}\label{nlc}
An operator $T:X \to X$ is called \emph{nonlinear $\mathfrak{D}$--contraction} if there is a $\mathfrak{D}$--function
$\vp$ satisfying $\vp(t)<t$ for $t>0$, and
\[
\|Tx-Tz\|\leqslant \vp(\|x-z\|),
\]
for all $x,z\in X$.
\end{definition}
The Definition~\ref{nlc} is taken from \cite{NW}. Originally, the concept of nonlinear contraction appears in \cite{BW}, in which the function $\vp$ is considered upper semicontinuous from the right and satisfying $\vp(t)<t$ for $t>0$.

Next, we recall some known concepts of compactness for nonlinear operators.
\begin{definition}
An operator $T \colon X\to X$ is said to be
\begin{itemize}
\item {\it Compact} if $T(X)$ is a relatively compact subset of $X$.
\item {\it Totally bounded} if $T(E)$ is a relatively compact subset of $X$ for any bounded subset $E\subset X$.
\item {\it Completely continuous} if it is continuous and totally bounded.
\end{itemize}	
\end{definition}
We next recall an extension of  Krasnoselskii's fixed--point theorem, which is essential to establish our
results in this section. This extension was achieved using an improvement of the Banach contraction mapping principle.
The reader can see \cite{BW} for a discussion about this issue.
\begin{theorem}\label{FPT}{\rm (\cite[Theorem~1]{NW})}
Let $S$ be a closed, convex, and bounded subset of a Banach space X and let $A : S \to X$, $B : S \to X$ be two
operators which satisfy the following conditions:
\begin{itemize}
\item[{\rm(a)}] $A$ is a nonlinear $\mathfrak{D}$--contraction.
\item[{\rm(b)}] $B$ is completely continuous.
\item[{\rm(c)}] $Au + Bv \in S$ for all $u,v\in S$.
\end{itemize}
Then the operator equation $Au + Bu = u$ has a solution in $S$.
\end{theorem}
{In the rest of this section, we assume that $J=[t_{0},t_{0}+a]$ with $t_{0} \in \real$ and  $a>0$, and the function  $g \colon J\to\real$ is nondecreasing
 and left--continuous. To state our results, we introduce a set of additional conditions about
$h,f \colon J \times \rn \to \rn$.}
\begin{description}
\item[(A0)] For every $t\in J$, the function $f(t,\cdot)$ is continuous.
\item[(A1)] For every $u \in \rn$, the function $f(\cdot, u)$ is Kurzweil--Stieltjes integrable w.r.t. $g$ on $J$.
\item[(A2)] There exists a function $M \colon J\to\real^{+}$ which is Kurzweil--Stieltjes integrable w.r.t. $g$ such that
\[
\left\|\int_{c}^{d} f(s, u) {\rm d}g(s)\right\|\leqslant\int_{c}^{d} M(s){\rm d}g(s),
\]
for all $u \in \rn$ and $[c,d] \subseteq J$.
\item[(A3)] For every $v \in \rn$, the function $h(\cdot, v)$ is regulated on $J$, and there exists a $\mathfrak{D}$--function
$\vp \colon \real^{+} \to \real^{+}$ with $\vp (t) < t$ such that
\[
\|h(t, u) - h(t, v)\| \leqslant \vp\left(\|u -v\|\right)
\]
for all $u, v \in \rn$, and $t\in J$.
\end{description}
\begin{remark} \label{R1} \rm
It is immediate to note  that condition (A3) implies that for a fixed $x\in G(J,\rn)$, the function $h(\cdot,x(\cdot))$ is regulated on $J$.
\end{remark}
In order to illustrate a situation in which  the condition (A1) is satisfied, we will present a property to ensure that a function
$s \mapsto f(s,x(s))$ is regulated on $J$ for each  $x\in G(J,\rn)$. For a function $x : J \to \rn$, we denote by
$Im(x) = \{x(t): t \in J\}$ the range set of $x$.

\begin{theorem} \label{thm3.1} Let $x\in G(J,\rn)$. Assume that $f : J\times \rn \to \rn$ is a function
such that for each $v \in \overline{Im(x)}$, the function $f( \cdot, v)$ is regulated on $J$. Let $y(s)  : J \to \rn$ the function given by $y(s)(t) = f(s,x(t))$. Assume that  $\{y(s)(\cdot): s\in J\}$
is an equiregulated set of functions. Then the function $z : J \to \rn$  given by $z(t) = f(t,x(t))$ is regulated.
\end{theorem}
\begin{proof} We will start by showing that the function $z$ has limits on the right. For this purpose, we will consider the set
$J_1=[t_0,t_0+a)$ and we define the function $L: J_{1} \to \rn$ by
\[
L(s) = \lim_{t \to s^{+}} f(t,x(s^{+})).
\]
Let $s_{0} \in J_{1}$. We claim that $\displaystyle \lim_{t\to s_{0}^{+}} z(t) = L(s_{0})$.
In fact, for $t\in J_1$, we can write
\[
z(t) - L(s_0) = z(t) - f(t,x(s_{0}^{+})) + f(t,x(s_{0}^{+})) - L(s_0).
\]
Since the functions $y(\xi) = f(\xi, x(\cdot))$  are equiregulated for $\xi \in J$, it follows from Definition~\ref{D3} that
 for every $\varepsilon > 0$, we can choose $\delta_{0} > 0$ such that
\[
\|f(\xi, v) - f(\xi, x(s_{0}^{+}))\| < \varepsilon/2,
\]
for all $\xi \in J$ and $v \in \rn$ such that $\|v - x(s_{0}^{+}) \| < \delta_{0}$. In addition, since $x(\cdot)$ is a regulated function,
there exists $\delta_1>0$ such that $\|x(t) - x(s_{0}^{+}) \| < \delta_{0}$ for $s_{0} < t <s_{0} + \delta_{1}$. Combining these assertions,
we infer that
\[
\|f(\xi, x(t)) - f(\xi, x(s_{0}^{+}))\| < \varepsilon/2
\]
for all  $\xi \in J$ and $t \in (s_{0}, s_{0} + \delta_{1})$. In particular, taking $\xi = t$, we obtain that
\[
\|z(t) - f(t, x(s_{0}^{+}))\| = \|f(t, x(t)) - f(t, x(s_{0}^{+}))\| < \varepsilon/2.
\]
Furthermore, since
${\displaystyle \lim_{t \to s_{0}^{+}} f(t,x(s_{0}^{+}))=L(s_0)}$, there exists $\delta_{2}>0$ such that
\[
\|f(t,x(s_{0}^{+})) - L(s_0)\| < \varepsilon/2,
\]
for all $s_0<t<s_0+\delta_2$.
Therefore, for  $\delta=\min\{\delta_1,\delta_2\}$ and  gathering our previous estimations, we obtain
\[
\|z(t) - L(s_0) \| < \varepsilon,
\]
for all $s_0 <t <s_{0} + \delta$. This shows that $\displaystyle \lim_{t \to s_{0}^{+}} f(t,x(t))= L(s_0)$. Proceeding in similar
way, we can prove that $\displaystyle \lim_{t \to s_{0}^{-}} f(t,x(t))$ exists for $s_{0} \in (t_{0}, t_{0} + a]$,
which completes the proof of the theorem.
\end{proof}

\begin{corollary} \label{C4}
 Assume that $f \colon J \times \rn \to \rn$ satisfies the conditions:
\begin{itemize}
\item[(a)] For every  $v \in \overline{Im(x)}$, the function $f( \cdot, v)$ is regulated on $J$.
\item[(b)] The functions $f(t, \cdot)$ are continuous uniformly for $t \in J$.
\end{itemize}
Let $x\in G(J,\rn)$. Then the function $z \colon J \to \rn$  given by $z(t) = f(t,x(t))$ is regulated.
\end{corollary}
\begin{proof} Applying Theorem~\ref{thm3.1} we only need to prove that the functions $f(s, x(\cdot))$ are equiregulated for
$s \in J$. Let $\ve > 0$. It follows from condition (b) that we can choose $\delta > 0$ such that
\[
\|f(s, u) - f(s, v)\| < \ve
\]
for all $s \in J$ and $u, v \in \rn$ such that $\|u -v\| < \delta$. Let $t \in [t_{0}, t_{0} + a)$. There exists $\delta_{1} > 0$
such that $\|x(t^{+}) - x(t^{\prime}) \| < \delta$ for all $t < t^{\prime} < t + \delta_{1}$. Collecting these assertions, we deduce that
\[
\|f(s, x(t^{+})) - f(s, x(t^{\prime}))\| < \ve
\]
for all $s \in J$ and  all $t < t^{\prime} < t + \delta_{1}$, which completes the proof.
\end{proof}
Returning to the more general context of integrable functions in the sense of Kurzweil-Stieltjes  and in order to establish our results, we recall the following property.

\begin{lemma} \label{L6}{\rm (\cite[Lemma~3.1]{MS})} Assume that the function $f$ satisfies the conditions {\rm(A0)-(A2)}. Let $x \colon J \to \rn$ be a regulated function.
Then the function $y(\cdot)$ given by $y(t) = f(t, x(t))$ for $t \in J$ is Kurzweil-Stieltjes integrable w.r.t. $g$ on $J$ and
\begin{equation}
\left\|\int_{c}^{d} f(s,x(s)) {\rm d} g(s) \right\| \leqslant \int_{c}^{d} M(s){\rm d}g(s), \label{equ3.5}
\end{equation}
for all interval  $[c,d] \subseteq J$.
\end{lemma}

In what follows, let us consider the following concept of solution for the problem
HMDE  \eqref{hmde1}--\eqref{hmde2}.
\begin{definition}
A regulated function $x \colon J \to \rn$ is said to be a {\it solution of problem HMDE \eqref{hmde1}--\eqref{hmde2}} if the following conditions hold:
\begin{itemize}
\item[(i)] The function $t \to h(t,x(t))$ is regulated on { $J$}.
\item[(ii)] The integral equation \eqref{hmde1} is verified, where the integral on the right--hand side is understood in the Kurzweil--Stieltjes sense.
\end{itemize}
\end{definition}
We are now in a position to establish a result on the existence of solutions for problem HMDE \eqref{hmde1}--\eqref{hmde2}.
\begin{theorem}\label{exist1}
Assume that hypotheses  {\rm(A0)-(A3)} hold. Assume further that
\begin{equation}
\liminf_{r \to \infty}\dfrac{\vp(r)}{r} < 1. \label{phi}
\end{equation}
Then there exists a solution for problem HMDE \eqref{hmde1}--\eqref{hmde2} defined on $J$.
\end{theorem}
\begin{proof}
 Using condition \eqref{phi}, we infer that there exists $N>0$ large enough such that
\[
\frac{\vp(N)}{N} + \frac{\|x_0-h(t_0,x_0)\| + H_{0} + K_{0}}{N} < 1,
\]
where $\displaystyle H_{0} =\sup\{\|h(t,0)\|:t\in J\}$ and $K_{0} =  \int_{t_0}^{t_0+a}M(s){\rm d}g(s)$.

Let $S$ be the subset of $G(J,\rn)$ defined by
\[
S:= \{x \in G(J,\rn): x(t_{0}) = x_{0}, \; \|x\|_{\infty} \leqslant N\}.
\]
Clearly, $S$ is a closed, convex and bounded subset of $G(J,\rn)$. We define the operators $A:S\to G(J,\rn)$ and $B:S\to G(J,\rn)$ by
\begin{equation}\label{A}
Ax(t)=h(t,x(t)), \; \;  t \in J,
\end{equation}
and
\begin{equation}\label{B}
Bx(t) = x_{0} - h(t_{0},x_{0}) + \int_{t_{0}}^{t} f(s,x(s)){\rm d} g(s), \; \;  t\in J.
\end{equation}
Using these definitions, the integral equation \eqref{hmde1} can be transformed into the operator equation
\begin{equation}\label{oe}
A x(t) + B x(t) = x(t), \; \;  t \in J.
\end{equation}
From (A3), it follows that the operator $A$ is a nonlinear $\mathfrak{D}$--contraction. In fact, for every $t\in J$, we have
\begin{equation*}
\|Ax(t)-Az(t)\|=\|h(t,x(t))-h(t,z(t))\|\leqslant\vp(\|x(t)-z(t)\|)\leqslant\vp(\|x-z\|_{\infty})
\end{equation*}
which implies that $\|Ax-Az\|_{\infty}\leqslant\vp(\|x-z\|_{\infty})$.
	
Next, we show that $B$ is a completely  continuous operator from $S$ into $G(J,\rn)$. First, we show
that $B$ is continuous on $S$. Let $(x_{n})_{n \in \mathbb{N}}$ be a sequence in $S$ converging to a function
$x\in S$. Let $d= \{\sigma_{0}, \sigma_{1},  \ldots, \sigma_{l} \}$ be a subdivision of the interval $[t_{0}, t_{0} +a]$.
Applying Lemma~\ref{L6}, for every finite subset $\{m_{1}, m_{2}, \ldots, m_{l}\}$ of $\mathbb{N}$, we have
\[
\left\|\sum_{j = 1}^{l} \int_{\sigma_{j -1}}^{\sigma_{j}} f(s, x_{m_{j}}(s)) {\rm d} g(s) \right\| \leqslant
 \sum_{j = 1}^{l} \int_{\sigma_{j -1}}^{\sigma_{j}} M(s) {\rm d} g(s) =  \int_{t_{0}}^{t_{0} +a} M(s) {\rm d} g(s)= K_0 < \infty,
\]
which allows us to use the dominated convergence theorem for the Kurzweil--Stieltjes integral.
Hence, combining condition  (A0) and  Theorem~\ref{DCT}, we obtain
\begin{equation*}
\begin{split}
\lim_{n\to +\infty}Bx_n(t)&=\lim_{n\to +\infty}\left[x_0-h(t_0,x_0)+\int_{t_0}^{t}f(s,x_n(s)){\rm d}g(s)\right]\\
&=x_0-h(t_0,x_0)+\lim_{n\to +\infty}\int_{t_0}^{t}f(s,x_n(s)){\rm d}g(s)\\
&= x_0-h(t_0,x_0)+\int_{t_0}^{t}f(s,x(s)){\rm d}g(s)=Bx(t),
\end{split}
\end{equation*}
for all $t\in J$. This shows that $B$ is a continuous map on $S$.
	
Next we prove that $B$ is a compact operator on $S$. As a consequence of  Lemma~\ref{L1}, we only need to show that  $B(S)$ is a uniformly bounded and equiregulated set in $G(J,\rn)$. In order to establish the first assertion,  for an arbitrary $x \in S$ we estimate
\begin{equation*}
\begin{split}
\|Bx(t)\|&\leqslant\|x_0-h(t_0,x_0)\|+\left\| \int_{t_0}^{t}f(s,x(s)){\rm d}g(s)\right\|\\
&\leqslant \|x_0-h(t_0,x_0)\|+K_0,
\end{split}
\end{equation*}
for all $t\in J$. Therefore, $\|Bx\|_{\infty} \leqslant \|x_0 - h(t_0,x_0)\|+ K_0$, for all $x\in S$. This shows that $B$ is
uniformly bounded on $S$.
	
On the other hand, let $t_{1}, t_{2} \in J$ with $t_{1} < t_{2}$. For any $x\in S$, we have
\begin{equation*}
\begin{split}
\|Bx(t_2)-Bx(t_1)\|=\left\|\int_{t_1}^{t_2}f(s,x(s))  {\rm d} g(s)\right\|\leqslant \int_{t_1}^{t_2}M(s){\rm d}g(s)=p(t_2)-p(t_1),
\end{split}
\end{equation*}
where $p(\cdot)$ is the function given by $p(t)=\int_{t_0}^{t} M(s) {\rm d} g(s)$ for $t\in J$. {Since,
by  Theorem~\ref{thm1.2} the} function $p$ is regulated on $J$, the above estimates implies that $B(S)$ is an equiregulated
set in $G(J,\rn)$.

To finish this proof, we show that condition (c) of Theorem~\ref{FPT} holds. Let $u, v \in S$ be arbitrary. Clearly,
$Au(t_0) + Bv(t_0) = h(t_0, u(t_0)) + x_0 - h(t_0, x_0) = x_0$. Also, we  estimate
\begin{equation*}
\begin{split}
\|Au(t)+ Bv(t)\|&\leqslant\|h(t,u(t))\|+\|x_0 -h(t_0,x_0)\|+\left\|\int_{t_0}^{t}f(s,v(s)){\rm d}g(s)\right\|\\
&\leqslant \|h(t,u(t))-h(t,0)\|+\|h(t,0)\|+ \|x_0 -h(t_0,x_0)\|+  \int_{t_0}^{t}M(s){\rm d}g(s)\\
&\leqslant \vp(N)+\|x_0 -h(t_0,x_0)\|+H_0+K_0<N,
\end{split}
\end{equation*}
for all $t\in J$. Consequently,  $\|Au+Bv\|_{\infty}\leqslant N$, {which in turn implies that}  $Au+Bv\in S$.
	
The preceding development shows that all the hypotheses of Theorem~\ref{FPT}  are satisfied, which allows us to conclude
that the equation \eqref{oe}  has a solution  in $S$. This shows that the problem HMDE \eqref{hmde1}--\eqref{hmde2} has a solution defined on $J$.
\end{proof}
In the next we will exhibit a simple application of Theorem~\ref{exist1}. To present this result, we need a previous property on
the Kurzweil-Stieltjes integration.
\begin{lemma} \label{L5} Assume that $g \colon [a, b] \to \real$ is  a nondecreasing left--continuous function and that
$\eta \colon [a, b] \to \real$ is a bounded Kurzweil--Stieltjes integrable w.r.t. $g$ function. Let $u \in G([a, b])$.
Then the function $u(\cdot) \eta(\cdot)$ is Kurzweil--Stieltjes integrable w.r.t. $g$.
\end{lemma}
\begin{proof} By \cite[Theorem~4.1.5]{MST} there exists a sequence
$(u_{n})_{n}$ of step functions which converges uniformly to $u$. It follows that each function
$u_{n}(\cdot) \eta(\cdot)$ is Kurzweil--Stieltjes integrable w.r.t. $g$. Since the sequence
$(u_{n} \eta)_{n}$  converges uniformly to $u \eta$, using \cite[Theorem~6.3.8]{MST} we conclude that the function
$u \eta$ is Kurzweil--Stieltjes integrable w.r.t. $g$.
\end{proof}
\begin{example}  \label{E1} \rm Consider the following   integral equation
\begin{equation}\label{exhmde1}
x(t)=\dfrac{1}{2}\sin^2(t)\ln(1+|x(t)|)+\int_{0}^{t}f(s,x(s)){\rm d}g(s), \; t\in[0,1],
\end{equation}
where $g$ is a nondecreasing left--continuous function, and $f \colon [0,1]\times\real\to\real$ is the function given
by $f(s, z) = \eta(s) e^{\gamma \cos(z)}$ for all $(s,z)\in [0,1]\times\real$, where  $\gamma>0$ is a constant and
$\eta:[0,1] \to \real^{+}$ is a bounded and   Kurzweil--Stieltjes integrable w.r.t. $g$ function.
It is clear that condition (A0) holds. Also, using Lemma~\ref{L5}, we obtain that condition (A1) holds.  Furthermore,
\[
\left| \int_{c}^{d} f(s, u)  {\rm d} g(s) \right| = \left| \int_{c}^{d} \eta(s) e^{\gamma \cos{u}}  {\rm d} g(s) \right|
\leqslant e^{\gamma} \int_{c}^{d} \eta (s)  {\rm d}g(s),
\]
for all $u \in \real$ and $[c,d] \subseteq[0,1]$, which shows  that  condition (A2) is fulfilled
with $M(s) = e^{\gamma} \eta(s) $ for $s \in[0,1]$.

In addition, let $h$ be the function given by $h(t, u)=\dfrac{1}{2}\sin^2(t) \ln(1+|u|)$ defined on $[0,1]\times\real$.
It is clear that $h$ is continuous, hence the function $h(\cdot,v)$ is regulated on $[0,1]$.  Moreover, for every $u, v \in \real$ we have
\begin{equation*}
|h(t,u)-h(t,v)| \leqslant\dfrac{1}{2} \left| \ln\left(\dfrac{(1+|u|)}{(1+|v|)}\right) \right|  \leqslant \vp(|u-v|),
\end{equation*}
where  $\vp(t)=\dfrac{1}{2}\ln(1+t)$ for $t\geqslant0$. Consequently, the function $h$ satisfies the condition (A3).

Since conditions (A0)--(A3) are fulfilled, and  $\displaystyle \liminf_{r \to \infty} \dfrac{\vp(r)}{r} = 0$, it follows
from Theorem~\ref{exist1} that the integral equation \eqref{exhmde1} has a solution defined on $[0,1]$.
\end{example}
Conditions (A2)  and (A3)  are very demanding since, said somewhat superficially, they tell us that the function $f$ is
uniformly bounded and that the function $h$ uniformly satisfies a Lipschitz--type property of continuity. However, these conditions can be generalized as follows.
\begin{description}
\item[(A2*)] There exists a function $M \colon J \times [0, +\infty) \to [0, +\infty)$ such that $M(\cdot, \al)$ is Kurzweil--Stieltjes integrable
w.r.t. $g$ for all $\al \geqslant 0$, the function $M(s, \cdot)$ is nondecreasing for all $s \in J$, and
\[
\left\|\int_{c}^{d} f(s, u) {\rm d}g(s) \right\| \leqslant \int_{c}^{d} M(s, \|u\|) {\rm d}g(s),
\]
for all $u \in \rn$, and $[c, d] \subseteq J$.
\item[(A3*)] For every $v \in \rn$, the function $h(\cdot, v)$ is regulated on $J$, and there exists a
function $\vp: \real^{+} \times \real^{+} \to \real^{+}$ such that for every $r \geqslant 0$, $\vp(\cdot, r)$ is a $\mathfrak{D}$--function satisfying
$\vp (t, r) < t$ for all $t > 0$,  and
\[
\|h(t, u) - h(t, v)\| \leqslant \vp\left(\|u -v\|, r \right)
\]
for all $u, v \in \rn$, $\|u\|, \|v\| \leqslant r$, and $t \in J$.
\end{description}
\begin{theorem}\label{T1} Assume that conditions {\rm(A0)}, {\rm(A1)}, {\rm(A2*)} and {\rm(A3*)} are fulfilled. Assume further that
\begin{equation}
\liminf_{r \to \infty} \left(\dfrac{\vp(r, r)}{r}+\frac{1}{r}  \int_{t_{0}}^{t_{0}+a} M(s, r)  {\rm d} g(s)\right) < 1. \label{equ3.1}
\end{equation}
Then there exists a solution for the HMDE \eqref{hmde1}--\eqref{hmde2} defined on $J$.
\end{theorem}
\begin{proof}
Since the proof follows the same lines as the one made for the Theorem~\ref{exist1}, we will present here only a brief sketch
of the proof.
\begin{itemize}
\item[{\rm (i)}] Using \eqref{equ3.1}, we {infer}  that there exists $N>0$ large enough such that
\begin{equation}
\vp(N, N) + H_{0} + \|x_{0} + h(t_{0},x_0)\| +  \int_{t_{0}}^{t_{0}+a} M(s,N)   {\rm d} g(s) <N. \label{equ3.6}
\end{equation}
We define $S = \{x \in G(J, \rn): x(t_{0}) = x_{0}, \; \|x\|_{\infty} \leqslant N\}$.
\item[{\rm (ii)}] As an immediate consequence of condition (A3*), we infer that the operator $A$ defined by \eqref{A} is a nonlinear
$\mathfrak{D}$--contraction with respect the $\mathfrak{D}$-function $\vp(\cdot, N)$  on $S$.
\item[{\rm (iii)}] In this step, we show that the operator $B$ defined by \eqref{B} is completely continuous.
Let $ (x_{n})_{n\in\mathbb{N}}$ be a sequence in $S$
that converges to $x$. Proceeding as in the proof of Theorem~\ref{exist1}, and as a consequence of (A2*), Lemma~\ref{L6} and 
 Theorem~\ref{DCT}, we obtain that  $B x_{n}(t) \to B x(t)$ as $n \to \infty$
uniformly for $t \in J$, which shows that $B$ is continuous.
Turning to apply {(A2*) and Lemma~\ref{L6}}, we infer that
\[
\left\|\int_{t_{0}}^{t} f(s, x(s))   {\rm d}g(s) \right\| \leqslant  \int_{t_{0}}^{t} M(s, N)   {\rm d}g(s) \leqslant 
\int_{t_0}^{t_0 + a} M(s, N){\rm d}g(s)
\]
for all $x \in S$. This implies that $B(S)$ is a bounded set. In addition, arguing as in the proof
of Theorem~\ref{exist1} we derive
that $B(S) \subset G(J, \mathbb R^n)$ is a relatively compact set.
\item[{\rm (iv)}] Let $u, v $ {be arbitrary functions} in $S$. Applying (A2*), (A3*), {Lemma~\ref{L6}}
and \eqref{equ3.6}, we can estimate
\begin{eqnarray*}
\|Au(t)+ B v(t)\| & \leqslant & \|h(t, u(t)) \| + \|x_{0} - h(t_{0}, x_{0})\| + \left\|\int_{t_{0}}^{t} f(s, v(s))  {\rm d}g(s) \right\| \\
& \leqslant &  \vp(N, N)+\|h(t,0)\|+ \|x_{0} - h(t_{0}, x_{0})\| +  \int_{t_{0}}^{t_{0} + a} M(s, N) {\rm d}g(s) \\
& \leqslant &  N
\end{eqnarray*}
for all $t \in J$, which implies that $Au+Bv\in S$.
\end{itemize}
As a consequence of steps (i)-(iv) we can affirm that all the hypotheses of Theorem~\ref{FPT} are satisfied, and thus this concludes the proof.
\end{proof}

\section{Differentiability and existence of $S$--asymptotically $\omega$--periodic solutions}

In this section, our {aim is to characterize the existence of special classes of solutions. In particular
we study the existence of continuous, differentiable and
 $S$--asymptotically $\omega$--periodic solutions for problem HMDE \eqref{hmde1}--\eqref{hmde2}.
We begin with a result that guarantees the existence of continuous solutions.
\begin{prop} \label{C2} Assume that conditions {\rm(A1), (A2)} and {\rm(A3*)} hold.
Assume further that  $g$ is continuous and
for every $v \in \rn$, the function $h(\cdot, v)$ is continuous. Then every  solution $x(\cdot)$ of problem HMDE
\eqref{hmde1}--\eqref{hmde2} is a  continuous function.
\end{prop}
\begin{proof} The proof of this statement is standard. For completeness of the text, we are going to include here an outline of it.
Let $x(\cdot)$ be a solution of   the problem HMDE \eqref{hmde1}--\eqref{hmde2} with $\|x\|_{\infty} \leqslant r$.
Suppose that $x(\cdot)$ is not continuous at
$\overline{t}$. Then there must be $\ve > 0$ and a sequence $(t_{k})_{k}$ that converges to $\overline{t}$ and such that
$\|x(t_{k}) - x(\overline{t})\| \geqslant \ve$ for all $k \in \mathbb{N}$.  It follows from expression \eqref{hmde1}
that
\begin{eqnarray*}
x(t_{k}) - x(\overline{t}) & = & h(t_{k}, x(t_{k})) - h(\overline{t}, x(\overline{t})) + \int_{t_{0}}^{t_{k}} f(s, x(s)) {\rm d}g(s)
- \int_{t_{0}}^{\overline{t}} f(s, x(s)) {\rm d}g(s) \\
& = & h(t_{k}, x(t_{k})) - h(t_{k}, x(\overline{t}) ) + h(t_{k}, x(\overline{t}) ) - h(\overline{t}, x(\overline{t}))
+ \int_{\overline{t}}^{t_{k}} f(s, x(s)) {\rm d}g(s)
\end{eqnarray*}
Without loss of generality, we assume that $t_{k} > \overline{t}$. From the previous decomposition it follows that
\begin{eqnarray}
\|x(t_{k}) - x(\overline{t})\| & \leqslant  & \|h(t_{k}, x(t_{k})) - h(\overline{t}, x(\overline{t}))\| + \left\| \int_{t_{0}}^{t_{k}} f(s, x(s)) {\rm d}g(s)
- \int_{t_{0}}^{\overline{t}} f(s, x(s)) {\rm d}g(s)\right\| \nonumber \\
& \leqslant  & \|h(t_{k}, x(t_{k})) - h(t_{k}, x(\overline{t}) ) \|+  \|h(t_{k}, x(\overline{t}) ) - h(\overline{t}, x(\overline{t})) \|
+ \left\| \int_{\overline{t}}^{t_{k}} f(s, x(s)) {\rm d}g(s) \right\| \nonumber \\
&\leqslant & \vp(\|x(t_{k}) -  x(\overline{t})\|, r) + \|h(t_{k}, x(\overline{t}) ) - h(\overline{t}, x(\overline{t})) \| +
\left\| \int_{\overline{t}}^{t_{k}} f(s, x(s)) {\rm d}g(s) \right\|. \label{equ4.2}
\end{eqnarray}
Since $h(\cdot, x(\overline{t}))$ is continuous, then $\|h(t_{k}, x(\overline{t}) ) - h(\overline{t}, x(\overline{t})) \| \to 0$ as $k \to \infty$,
and, using Theorem~\ref{thm1.2}, ${\displaystyle \int_{\overline{t}}^{t_{k}} f(s, x(s)) {\rm d}g(s) \to 0}$ as $k \to \infty$.
Since $\vp([\ve, 2r], r)$ is a compact set, we can choose $\delta > 0$ such that $\delta \leq t - \vp(t, r)$ for all $t \in [\ve, 2r]$. Using that
$ \|x(t_{k}) -  x(\overline{t})\| \in [\ve, 2r]$, and applying the estimate  \eqref{equ4.2}, we obtain
\[
\|x(t_{k}) - x(\overline{t})\| \leqslant \|x(t_{k}) -  x(\overline{t})\| - \delta +
\|h(t_{k}, x(\overline{t}) ) - h(\overline{t}, x(\overline{t})) \| +
\left\| \int_{\overline{t}}^{t_{k}} f(s, x(s)) {\rm d}g(s) \right\|
\]
which implies that
\[
\delta \leqslant \|h(t_{k}, x(\overline{t}) ) - h(\overline{t}, x(\overline{t})) \| +
\left\| \int_{\overline{t}}^{t_{k}} f(s, x(s)) {\rm d}g(s) \right\| \to 0, \; k \to \infty.
\]
This contradiction shows that $x(\cdot)$ is continuous at $\overline{t}$.
\end{proof}}
Next we study the differentiability of a solution of problem HMDE \eqref{hmde1}--\eqref{hmde2}. Before stating our result,
we will introduce a notation. Assume that $h \colon [t_{0}, t_{0} + a] \times \rn \to \rn$ is differentiable and
$h(t, x) =col (h_{1}, \ldots, h_{n})$. We denote $J_{x}h(p)$ to the Jacobian matrix
\[
J_{x}h(p) =  \left( \begin{array}{cccc}
\frac{\partial h_{1} (p)}{\partial x_{1}} & \frac{\partial h_{1} (p)}{\partial x_{2}} & \cdots & \frac{\partial h_{1} (p)}{\partial x_{n}} \\
& & & \\
\frac{\partial h_{2} (p)}{\partial x_{1}} & \frac{\partial h_{2} (p)}{\partial x_{2}} & \cdots & \frac{\partial h_{2} (p)}{\partial x_{n}} \\
\vdots & \vdots & & \vdots \\
\frac{\partial h_{n} (p)}{\partial x_{1}} & \frac{\partial h_{n} (p)}{\partial x_{2}} & \cdots & \frac{\partial h_{n} (p)}{\partial x_{n}}
\end{array} \right).
\]
Furthermore,
\[
h(t, x) = h(t_{0}, x_{0}) + \frac{\partial h(t_{0}, x_{0}) }{\partial t} (t - t_{0}) + J_{x}h(t_{0}, x_{0}) (x - x_{0}) + R(t, x, t_{0}, x_{0})
\]
where ${\displaystyle \lim_{t \to t_{0}, x \to x_{0}} \frac{R(t, x, t_{0}, x_{0})}{\|(t, x) - (t_{0}, x_{0})\|} = 0}$. 

On the other hand, for a function $u(\cdot)$ defined on an interval $J$, we denote by $u_{+}^{\prime}(t)$ (and $u_{-}^{\prime}(t)$) the right (respectively, left) derivative of $u$ at $t$, when these derivatives exist.

\begin{prop} \label{C3} Assume that conditions {\rm(A1), (A2)} and {\rm(A3*)} hold.
Assume further that  $h$ is continuously differentiable, with $I - J_{x}h(t, x(t))$ invertible, and that $f$ and $g$ are
continuous functions. Let $x(\cdot)$ be a solution of problem
HMDE \eqref{hmde1}--\eqref{hmde2}.
Then  $x(\cdot)$ satisfies the following properties:
\begin{itemize}
\item[(a)] The function $x(\cdot)$ is differentiable a.e. for $t \in [t_{0}, t_{0} + a]$.
\item[(b)] For each $t \in [t_{0}, t_{0} + a)$ such that $g_{+}^{\prime}(t)$ exists, there is also the right derivative $x_{+}^{\prime}(t)$.
Similarly, for each $t \in (t_{0}, t_{0} + a]$ such that $g_{-}^{\prime}(t)$ exists, there is also the left derivative $x_{-}^{\prime}(t)$.
\item[(c)] If $g$ is differentiable at $t$, then the function $x(\cdot)$ is also differentiable at  $t$.
\item[(d)] If $g$ is continuously differentiable at $t$, then the function $x(\cdot)$ is also continuously differentiable at  $t$.
\end{itemize}
\end{prop}
\begin{proof} It follows from Corollary~\ref{C2} that $x(\cdot)$ is a continuous function. We begin the proof with a simple remark. Let
${\displaystyle u(t) = \int_{t_{0}}^{t} f(s, x(s)) {\rm d} g(s)}$. Assume initially that the derivative $g^{\prime}(t)$ exists.
Then the derivative $u^{\prime}(t)$ also exists and $u^{\prime}(t)= f(t, x(t)) g^{\prime}(t) $. In fact, under our hypotheses, the
Kurzweil-Stieltjes integral   ${\displaystyle  \int_{t_{0}}^{t} f(s, x(s)) {\rm d} g(s)}$
coincides with the Riemann-Stieltjes integral of $f(s, x(s))$ w.r.t. $g$ (\cite[Theorem~6.2.12]{MST}) and the assertion is a well known
property  for the Riemann-Stieltjes integral.
It follows from \eqref{hmde1}  that
\[
\frac{x(t + \Delta) - x(t)}{\Delta}= \frac{\partial h(t, x(t)) }{\partial t}  + J_{x}h(t, x(t)) \frac{x(t + \Delta) - x(t)}{\Delta} +
\frac{R(t + \Delta, x(t + \Delta), t, x(t))}{\Delta} + \frac{u(t+\Delta) - u(t)}{\Delta},
\]
which implies that
\[
\frac{x(t + \Delta) - x(t)}{\Delta}= \left(I-  J_{x}h(t, x(t)) \right)^{-1}
\left(\frac{\partial h(t, x(t)) }{\partial t}  +  \frac{u(t+\Delta) - u(t)}{\Delta} + \frac{R(t + \Delta, x(t + \Delta), t,x(t))}{\Delta}\right).
\]
Hence, taking the limit above as $\Delta \to 0$, and using the previous remark, we obtain
\begin{equation}
x^{\prime}(t) = \left(I-  J_{x}h(t, x(t)) \right)^{-1} \left(\frac{\partial h(t, x(t)) }{\partial t}  +  f(t, x(t)) g^{\prime}(t) \right). \label{equ3.9}
\end{equation}
Since $g$ is a nondecreasing function, the derivative $g^{\prime}(t)$  exists a.e. Thus, the assertion (a) is an immediate
consequence of the  preceding development.  Claims (b) and (c) are similarly demonstrated. Finally, statement (d) is a consequence
of identity \eqref{equ3.9}.
\end{proof}
We next are concerned with  the existence of solutions on the interval $[0, \infty)$.
Taking into account that our later objective is to establish the existence of approximately periodic solutions, {initially}
we will restrict  ourselves to investigate the existence of bounded solutions on $[0, \infty)$. We denote by $G_{b}([0,\infty), \rn)$ the space consisting of all bounded regulated functions $u: [0,\infty) \to \rn$ endowed with the norm of uniform convergence. We slightly modify condition (A1), as follows:
\begin{description}
\item[(A1*)] For every {$u  \in   \rn$ and every $a > 0$, the function $f(\cdot, u)$} is Kurzweil--Stieltjes integrable w.r.t. $g$
on the interval $[0, a]$.
\end{description}

{Furthermore, from here on, the condition (A3*) should be understood as follows: for every $v \in \rn$, the function $h(\cdot, v)$ belongs to $G_{b}([0,\infty),\rn)$ and there exists a function $\vp: \real^{+} \times \real^{+} \to \real^{+}$ such that for every $r \geqslant 0$, $\vp(\cdot, r)$ is a $\mathfrak{D}$--function satisfying
$\vp (t, r) < t$ for all $t > 0$,  and
\[
\|h(t, u) - h(t, v)\| \leqslant \vp\left(\|u -v\|, r \right)
\]
for all $u, v \in \rn$, $\|u\|, \|v\| \leqslant r$, and $t \geqslant0$.}

As usual, we write ${\displaystyle  \int_{0}^{\infty} M(s, N) {\rm d} g  = \lim_{t \to \infty} \int_{0}^{t} M(s, N) {\rm d} g}$.
\begin{theorem} \label{thm3.2} Assume that conditions {\rm(A0), (A1*), (A2*)} and {\rm(A3*)} hold. Suppose further that
\begin{equation}
\liminf_{N \to \infty} \frac{1}{N} \left(  \sup_{0  \leqslant t < \infty, \|u\| \leqslant N} \|h(t, u)\| +
 \int_{0}^{\infty} M(s, N) {\rm d} g \right) < 1. \label{equ3.3}
\end{equation}
Then there exists a bounded solution for problem HMDE \eqref{hmde1}--\eqref{hmde2} on $[0, \infty)$.
\end{theorem}
\begin{proof} It follows from Theorem \ref{T1} that problem HMDE \eqref{hmde1}--\eqref{hmde2} has a solution on every
interval of type $[t_{0}, t_{0} +a]$.
Using this observation, we will carry out an inductive construction to prove the existence of solution on $[0, \infty)$.
Let us consider the intervals of the form $[k-1, k]$ for $k \in \mathbb N$. Let $k = 1$. By Theorem~\ref{T1}, we know that there exists a solution,
denoted by $x^{1}(\cdot)$, of  problem HMDE \eqref{hmde1}--\eqref{hmde2} defined on the interval $[0, 1]$.
Suppose now that $x^{k}$ for $k \in \mathbb{N}$ is a solution of problem  HMDE \eqref{hmde1}--\eqref{hmde2} on $[k -1, k]$.
Proceeding as in the definition of $x^{1}(\cdot)$, let $x^{k +1}(\cdot)$ be a solution of the problem
\begin{eqnarray*}
x(t) & = &  x(k) + h(t, x(t)) - h(k, x(k)) + \int_{k}^{t} f(s, x(s)){\rm d}g(s), \; k \leqslant t \leqslant k +1, \\
x(k)& = & x^{k}(k) \in \rn.
\end{eqnarray*}
Let us define the function $x(t) = x^{k}(t)$ for $t \in [k-1, k]$. It is clear that $x(\cdot)$ is a solution of problem
HMDE \eqref{hmde1}--\eqref{hmde2} on $[0, \infty)$.

Assume now that $x(\cdot)$ is not bounded on $[0, \infty)$. Let $N(t)  = \displaystyle \sup_{0 \leqslant s \leqslant t} \|x(s)\|$.
It is clear that $N(t) \to \infty$ as $t \to \infty$. Hence, we can select $t \geqslant t_{1}$ such that $N(t) \geqslant N_{1}>0$.
We can assume that ${\displaystyle \frac{1}{N_{1}} \|x_{0} - h(t_{0}, x_{0})\| \leqslant \ve/2}$, for $\varepsilon>0$ given.
Since $x$ is a solution, it follows that
\begin{eqnarray*}
\|x(t)\| & \leqslant & \|x_{0} - h(t_{0}, x_{0})\| + \|h(t, x(t)) \|+ \left\| \int_{0}^{t} f(s,x(s))  {\rm d} g(s) \right\| \\
& \leqslant &  \|x_{0} - h(t_{0}, x_{0})\| + \sup_{0 \leqslant \xi \leqslant t, \|u\| \leqslant N(t)} \|h(\xi, u) \| + \int_{0}^{t} M(s, N(t))
{\rm d} g(s).
\end{eqnarray*}
Hence, using \eqref{equ3.3} we obtain that
\[
1 \leqslant \frac{1}{N(t)} \|x_{0} - h(t_{0}, x_{0})\| + \frac{1}{N(t)} \sup_{\xi < \infty, \|u\| \leqslant N(t)} \|h(\xi, u) \| +
\frac{1}{N(t)} \int_{0}^{t} M(s, N(t)) {\rm d} g(s) < 1- \ve/2
\]
which is a contradiction. Thus, we complete the proof that $x(\cdot)$ is a bounded solution of problem
HMDE \eqref{hmde1}--\eqref{hmde2} on $[0, \infty)$.
\end{proof}
In the remaining of this section, we will study the existence of a type of solution that resembles a $\omega$--periodic function.
In \cite{HPT1,HPT2}, it was introduced the concept of S--asymptotically $\om$--periodic functions
for bounded continuous functions defined on $[0, \infty)$,
and the authors studied the main properties of this class of functions. In what follows, we will extend the definition
considered in those works to include regulated functions.
\begin{definition} \label{D1} A bounded regulated function $x : [0, \infty) \to X$ is called \emph{S--asymptotically $\om$--periodic}
if there exists $\om>0$ such that
\[
x(t + \om) - x(t) \to 0, \; t \to \infty.
\]
\end{definition}
{Before establishing the concepts necessary to study this topic}, inspired by \cite[Example~3.2]{HPT1}, we exhibit an example
to illustrate the concept of S--asymptotically $\om$--periodic for regulated functions.
\begin{example} \label{E3} Let $(a_{n})_{n\in\mathbb{N}_0} \subset \real $ be a bounded sequence of real numbers { consisting} of numbers different
from each other and such that $a_{n +1} - a_{n} \to 0$ as $n \to \infty$.
Let $x: [0, \infty) \to \mathbb R$ be the function given  by
\[
x(t) =\begin{cases}
a_1 + (a_1 - a_0)(t-1), & \text{if} \  0 \leqslant t \leqslant 1,\\
a_{n+1} + (a_{n+1} - a_{n-1})(t-n-1), & \text{if} \ n < t \leqslant n+1, \; n\in\mathbb{N}.
\end{cases}
\]
Clearly the function $x$ is continuous on every subinterval of type $(n,n+1)$ with $n\in\mathbb{N}_0$
and left--continuous at every $t=n \in\mathbb{N}$. Moreover, $x(n)=a_{n}$ and
$\Delta^{+}x(n)=a_{n-1}-a_{n}$ for all $n\in\mathbb{N}$. Consequently,  $x(\cdot)$ is a regulated
function with a countable  set of discontinuities on $[0,\infty)$ and $x(\cdot)$ is not periodic. The
boundedness of $x(\cdot)$ follows
from the properties of the sequence $(a_{n})_{n\in\mathbb{N}}$.
In addition, for every $t\in (n,n+1]$ we obtain
\[
|x(t+1)-x(t)|\leqslant 2|a_{n+2} - a_{n+1}|+|a_{n} - a_{n-1}|.
\]
Therefore, $\displaystyle\lim_{t\to\infty}[x(t+1)-x(t)]=0$, which shows that the function $x$ is $S$--asymptotically $1$--periodic.
\end{example}
We denote by $S_{\om}(X)$ the space consisting of all  S--asymptotically $\om$--periodic functions from $[0, \infty)$ to $X$. It is clear that
$S_{\om}(X)$ is a Banach space provided with the norm of uniform convergence.
\begin{definition} \label{D2} Let  $p : [0, \infty) \times \rn \to \rn $ be a function that satisfies the following conditions:
\begin{itemize}
\item[(i)] For every $u \in \rn$, the function $p(\cdot, u)$ is bounded and regulated.
\item[(ii)] For every $t \geqslant 0$, the function $p(t, \cdot)$ is continuous.
\end{itemize}
We say that $p$ is \emph{uniformly S--asymptotically $\om$--periodic on bounded sets} if there exists $\om>0$ such that
\[
p(t + \om, u) - p(t, u) \to 0, \; t \to \infty,
\]
uniformly for $u$ in bounded subsets of $\rn$.
\end{definition}
To establish our result on the existence of S--asymptotically $\om$--periodic solutions of problem HMDE \eqref{hmde1}--\eqref{hmde2},
we need to develop some previous properties.
\begin{lemma} \label{L2} Let $x \in S_{\om}(\rn)$. Assume that $p$ is a uniformly S--asymptotically $\om$--periodic
on bounded sets function that satisfies condition {\rm(A3*)}. Then the function $t \mapsto p(t, x(t))$ is
S--asymptotically $\om$--periodic.
\end{lemma}
\begin{proof} From the decomposition
\[
p(t + \om, x(t + \om)) - p(t, x(t)) =  p(t + \om, x(t + \om)) - p(t , x(t + \om)) + p(t, x(t + \om)) - p(t, x(t)),
\]
we infer that
\begin{align*}
\|p(t + \om, x(t + \om)) - p(t, x(t))\|  \leqslant &
   \|p(t + \om, x(t + \om)) - p(t , x(t + \om))\| + \|p(t, x(t + \om)) - p(t, x(t))\| \\
 \leqslant &   \|p(t + \om, x(t + \om)) - p(t , x(t + \om))\| + \vp(\|x(t + \om) -  x(t) \|, \|x\|_{\infty}).
\end{align*}
Since $\|p(t + \om, x(t + \om)) - p(t , x(t + \om))\| \to 0$ as $t \to \infty$, because $p$ is
uniformly S--asymptotically $\om$--periodic
on bounded sets function, and $\vp(\|x(t + \om) -  x(t) \|, \|x\|_{\infty}) \to 0$ as $t \to \infty$ by the
properties of $\vp(\cdot, \|x\|_{\infty})$, we conclude
that $\|p(t + \om, x(t + \om)) - p(t, x(t))\| \to 0$ as $t \to \infty$.
\end{proof}

\begin{definition} \label{D4} Let  $y_{k} : [0, \infty)  \to X $, $k \in \mathbb{N}$, be regulated functions.
The sequence $(y_{k})_{k}$ is said to be \emph{uniformly S-asymptotically $\om$--periodic} if for every $\ve > 0$ there exists
$T_{\ve} > 0$ such that
\[
\|y_{k}(t + \om) - y_{k}(t) \| \leqslant \ve
\]
for all $t \geqslant T_{\ve}$ and $k \in \mathbb{N}$.
\end{definition}

{In what follows, until the end of this section, when we write that a sequence $(y_k)_{k}\subset G([0,\infty),\rn)$ converges to $y\in G([0,\infty),\rn)$ for the Fr\'echet topology on $[0,\infty)$, we mean that $(y_k)_{k}$ is uniformly convergent to $y$ on bounded subsets of $[0,\infty)$. For a detailed discussion about the Fr\'echet topology on the space of regulated functions on unbounded intervals, we refer the reader to  \cite{DO}.}

\begin{lemma} \label{L4} Let $(y_{k})_{k}$ be a bounded uniformly S--asymptotically $\om$--periodic sequence in $ S_{\om}(X)$.
Assume that $(y_{k})_{k}$ converges to $y$ {for the Fr\'echet topology} on $[0, \infty)$. Then $y \in S_{\om}(X)$.
\end{lemma}
\begin{proof} Let $\ve > 0$. It follows from our hypotheses that there exists $T_{\ve} > 0$ such that
\[
\|y_{k}(t + \om) - y_{k}(t) \| \leqslant \ve/3
\]
for all $t \geqslant T_{\ve}$ and $k \in \mathbb{N}$. In addition, for each $t \geqslant T_{\ve}$ there exists $k_1 \in \mathbb{N}$ such that
$\|y(s) - y_{k}(s) \| \leqslant \ve/3$ for all $k \geqslant k_1$ and for all $s \in [0, t +\om]$. Combining these estimates, we have
\begin{eqnarray*}
\|y(t + \om) - y(t) \| & \leqslant & \|y(t + \om) -y_{k}(t + \om)\| + \|y_{k}(t + \om) - y_{k}(t) \| + \|y_{k}(t) - y(t) \| \\
& \leqslant & \ve
\end{eqnarray*}
for all $t \geqslant T_{\ve}$. This implies that $y \in S_{\om}(X)$.
\end{proof}
\begin{theorem} \label{thm3.3} Assume that the hypotheses  of Theorem~\ref{thm3.2} hold. Assume in further that
the following conditions hold:
\begin{itemize}
\item[(a)] The function $h$ is uniformly S--asymptotically $\om$--periodic on bounded sets.
\item[(b)] For every $N > 0$, the function $\mu(\cdot, N)$ given by $\mu(t, N) = \int_{0}^{t} M(s, N) {\rm d}  g(s)$, $ t \geqslant 0$,
is S--asymptotically $\om$--periodic.
\end{itemize}
Then problem HMDE \eqref{hmde1}--\eqref{hmde2} on $[0, \infty)$ has an S--asymptotically $\om$--periodic solution.
\end{theorem}
\begin{proof} We consider the maps $A$ and $B$ given by \eqref{A} and \eqref{B}, respectively.
Using \eqref{equ3.3}, we select  $N > 0$ large enough such that
\begin{equation}
\|x_{0} - h(0, x_{0}) \| + \sup_{0 \leqslant t < \infty, \|u\| \leqslant N} \|h(t, u)\| + \int_{0}^{\infty} M(s, N) {\rm d} g(s) \leqslant N. \label{equ3.7}
\end{equation}
We divide the proof in five parts.

{\bf Step (i).} Let $V_{N} = \{x \in G([0, \infty), \rn): x(0)=x_0, \; \|x\|_{\infty} \leqslant N\}$. {It is clear that $V_{N}$ is a
bounded convex closed subset of $G_{b}([0, \infty), \rn)$}. Arguing as in \cite[Theorem~2.3]{DHAGE},
we will show that $(I-A)^{-1}B : V_{N} \to V_{N}$ is well defined.
At first, we observe that for each $y \in V_{N}$ there exists at most one $x \in V_{N}$ such that
$x = Ax + By$. In fact, if we assume that $x_{1}, x_{2}$ satisfy the condition with $x_{1}(t) \neq x_{2}(t)$ for some
$t > 0$,  then $x_{1} - x_{2} = A x_{1} - A x_{2}$. This implies that
\[
\|x_{1}(t) - x_{2}(t)\| = \|A x_{1}(t) - A x_{2}(t)\| \leqslant \vp(\|x_{1}(t) - x_{2}(t)\|) < \|x_{1}(t) - x_{2}(t)\|
\]
which is a contradiction. In addition, for a fixed $y \in V_{N}$, we define the map $F: V_{N} \to
G([0, \infty), \rn)$ by $F x = A x + B y$. It follows from \eqref{equ3.7} that $\|Fx(t)\| \leqslant N$ for all $t \geqslant 0$. Consequently,
we can affirm that $F: V_{N} \to V_{N}$. {Since $F$ is a nonlinear $\mathfrak{D}$--contraction,  applying \cite[Theorem~1]{BW}, we infer that $F$ has a unique fixed point $x$ in
$V_{N}$}. It is clear that $x = (I-A)^{-1}B y$.

{\bf Step (ii).} In what follows, and unless we specify somewhat different, we will consider the space
$G_{b}([0, \infty), \rn)$ {provided with the Fr\'echet topology}, and the subsets of $G_{b}([0, \infty), \rn)$ provided with the topology
induced by $G_{b}([0, \infty), \rn)$.
We affirm that the map  $(I-A)^{-1} : B(V_{N}) \to V_{N}$ {is continuous for the Fr\'echet topology}. In fact, assume that
the sequence $(B y_{j})_{j}$ converges  uniformly on bounded intervals. Let $x_{j} = A x_{j} + B y_{j}$ and $ a > 0$.
If we assume that $(x_{j})_{j}$ is not a Cauchy sequence, we derive the existence of $\ve > 0$, $t_{k} \in [0, a]$,
$m_{k}, j_{k} \to \infty$ such that $\|x_{j_{k}}(t_{k}) - x_{m_{k}}(t_{k})\| \geqslant \ve$. In what follows, we
will show that this statement is not possible.   There exists $\delta > 0$ such that $\vp(\al) \leqslant \al - \delta$ for
$\ve \leqslant \al \leqslant 2 N$.
Since $\|x_{j}(t)\| \leqslant N$ {for all $t \geqslant 0$}, we can apply the previous assertion with $\al =
\|x_{j_{k}}(t_{k}) - x_{m_{k}}(t_{k})\|$.
Therefore,
\begin{eqnarray*}
\|x_{j_{k}}(t_{k}) - x_{m_{k}}(t_{k})\| & \leqslant & \|A x_{j_{k}}(t_{k}) - A x_{m_{k}}(t_{k})\| + \|B y_{j_{k}}(t_{k}) - B y_{m_{k}}(t_{k})\|  \\
& \leqslant & \vp(\|x_{j_{k}}(t_{k}) - x_{m_{k}}(t_{k})\|)  + \|B y_{j_{k}}(t_{k}) - B y_{m_{k}}(t_{k})\|  \\
& \leqslant & \|x_{j_{k}}(t_{k}) - x_{m_{k}}(t_{k})\| - \delta  + \|B y_{j_{k}}(t_{k}) - B y_{m_{k}}(t_{k})\|.
\end{eqnarray*}
Hence, we have
\[
\delta  \leqslant \|B y_{j_{k}}(t_{k}) - B y_{m_{k}}(t_{k})\| \to 0, \; k \to \infty,
\]
which is a contradiction, {and this proves the} statement.

{\bf Step (iii).} Let $x(\cdot) \in S_{\om}(\rn) \cap V_{N}$. We will show that
$A x, B x \in S_{\om}(\rn)$. It follows from Lemma~\ref{L2} that $t \mapsto h(t, x(t))$ is an
S-asymptotically $\om$-periodic function. In addition, let ${\displaystyle v(t) = \int_{0}^{t} f(s, x(s)) {\rm d} g(s)}$.
Then
\[
v(t + \om) - v(t) =  \int_{t}^{t + \om} f(s, x(s)) {\rm d} g(s),
\]
which implies that
\begin{equation}
\|v(t + \om) - v(t)\| =  \left\| \int_{t}^{t + \om} f(s, x(s)) {\rm d} g(s) \right\| \leqslant \int_{t}^{t + \om} M(s, N) {\rm d} g(s) \to 0, \;
t \to \infty. \label{equ4.3}
\end{equation}

{\bf Step (iv).}  In this step, we will show that $B \colon S_{\om}(\rn) \cap V_{N} \to S_{\om}(\rn) \cap V_{N}$
{is completely continuous for the Fr\'echet topology}.
First, let $(x_{k})_{k}$ be a sequence that converges to $x$ uniformly on bounded intervals. Applying Theorem~\ref{DCT}
to the sequence $f(s, x_{k}(s)) - f(s, x(s))$, we conclude that $B x_{k} (t) \to B x(t) $ as $k \to \infty$ uniformly on intervals of type
$[0, a]$ for all $a > 0$. This shows that $B$ is continuous for the {Fr\'echet topology}.
Let $(x_{k})_{k}$ be a sequence in $S_{\om}(\rn) \cap V_{N}$.
Proceeding as in the proof of Theorem~\ref{exist1}, {we can affirm that $B \colon G([0, a], \rn) \to G([0, a], \rn)$ is completely
continuous for all $a > 0$. Consequently},
there exists a subsequence  $(x_{k}^{1})_{k}$ of $(x_{k})_{k}$ such that $(B x_{k}^{1})_{k}$ converges uniformly to $y^{1}$
in the interval $[0, 1]$. Repeating inductively this argument, we can affirm that there exists a subsequence  $(x_{k}^{j+1})_{k}$
of $(x_{k}^{j})_{k}$  such that $(B x_{k}^{j+1})_{k}$ converges uniformly to $y^{j+1}$ in the interval $[0, j+1]$.
It is clear that $y^{j+1}(t) = y^{j}(t)$ for $t \in [0, j]$. This allows us to define $y(t) = y^{j}(t)$ for $t \in [0, j]$.
In addition, using a diagonal selection process, we know that there exists a sequence $(B x_{k_{j}}^{j})_{j}$
that converges to $y$ uniformly on compact intervals. {Repeating the estimate \eqref{equ4.3}}, we deduce that
the sequence $(B x_{k_{j}}^{j})_{j}$ is uniformly S--asymptotically $\om$--periodic. Applying Lemma~\ref{L4}, we can affirm that
$y \in S_{\om}(\rn) \cap V_{N}$.

{\bf Step (v).} We define  $W = \{x \in S_{\om}(\rn): x(0) = x_{0}, \; \|x\|_{\infty} \leqslant N\}$ endowed with the {Fr\'echet topology}.
For $y \in W$, let $x  \in V_{N}$ such that $x = A x + By$. This implies that
\begin{eqnarray*}
\lefteqn{x(t + \om) - x(t)  =  A x(t + \om) - A x(t) + B y(t + \om) - B y(t)} \\
& = & h(t +\om,  x(t + \om)) - h(t, x(t)) + B y(t + \om) - B y(t) \\
& = & h(t +\om,  x(t + \om)) - h(t,  x(t + \om)) + h(t,  x(t + \om)) - h(t, x(t))  + B y(t + \om) - B y(t)
\end{eqnarray*}
which yields that
\begin{eqnarray}
\lefteqn{ \|x(t + \om) - x(t)\| \leqslant  \|h(t +\om,  x(t + \om)) - h(t,  x(t + \om))\| + \|h(t,  x(t + \om)) - h(t, x(t))\|} \nonumber \\
& & +  \|B y(t + \om) - B y(t)\| \nonumber \\
& \leqslant & \|h(t +\om,  x(t + \om)) - h(t,  x(t + \om))\| + \vp(\|x(t + \om) -  x(t)\|,N)
 + \|B y(t + \om) - B y(t)\|. \label{equ3.8}
\end{eqnarray}
Next, we analyze each term on the right hand { of \eqref{equ3.8}}. The first term converges to zero when $t$ goes to infinity, because
the function $h$  is  uniformly S--asymptotically $\om$--periodic on bounded sets. The term
$\|B y(t + \om) - B y(t)\| \to 0$ as $t \to \infty$ by the assertion in Step (iii). Arguing as in Step (ii), from
\eqref{equ3.8} we deduce that $\|x(t + \om) - x(t)\| \to 0$ as $t \to \infty$. {Hence, $x \in W$ and we can define
the map  $U : W \to W$  given by  $U y = (I -A)^{-1} B y$.
Combining steps (ii) and (iv) we derive that $U$ is a completely continuous map. Moreover, we can show that $\overline{U(W)} \subseteq W$. In fact,
for $y \in \overline{U(W)}$ we can select a sequence $(x_{k})_{k}$ in $W$ such that $U(x_{k}) \to y$ as $k \to \infty$. It follows from Step (iv) that there exists a subsequence of $(x_{k})_{k}$, that we will continue to denote with the same index, such that $Bx_{k} \to z$ as $k \to \infty$ with $z \in W$.
Applying again Step (ii), we obtain that $U(x_{k}) = (I - A)^{-1} Bx_{k} \to (I - A)^{-1} z \in V_{N}$ as $k \to \infty$. This implies that
$y =  (I - A)^{-1} z$ and $y = Ay + z$. Arguing as before to get the estimates \eqref{equ3.8}, we infer that $y \in S_{\om}(\rn)$, which
in turn implies that $y \in W$. Applying} Schauder--Tychonoff's Theorem
\cite[Theorem~II.7.1.13]{GD}, we obtain that $U$ has a fixed point $x$ in $W$. Therefore, $x$ is a fixed point of $A + B$ and
$x(\cdot) \in S_{\om}(\rn)$ is a solution of problem HMDE \eqref{hmde1}--\eqref{hmde2} on $[0, \infty)$.
\end{proof}

\section{Continuous dependence on parameters}
In this section, we present a continuous dependence result on parameters for the problem HMDE \eqref{hmde1}--\eqref{hmde2} 
on $J = [t_{0}, t_{0} +a]$ for functions $h, f : J \times \rn \to \rn$ and $g: J \to \real$. We assume that $f$
satisfies conditions (A0)-(A1) and that $g$ is a nondecreasing left--continuous function. To simplify the presentation, we will
restrict ourselves to consider the type of
conditions considered in Theorem~\ref{exist1}.{ Hence, in  order to present our statement,
let $h_{k}, f_{k} : J \times \rn \to \rn$, $k \in \mathbb{N}$, be  functions which satisfy  conditions
(A0)--(A3) with respect to functions $\vp_{k}:\real^{+}\to\real^{+}$, $M_{k}: J \to\real^{+}$, for $k \in \mathbb{N}$, respectively. We denote by $A_{k}$ and $B_{k}$ the operators
associated to functions $h_{k}, f_{k}$ by mean of expressions \eqref{A} and \eqref{B}, respectively. We assume that
$\vp_{k}$ are $\mathfrak{D}$--functions that satisfy the condition $\vp_{k}(t) < t$ uniformly on bounded intervals.
To specify this requirement, we introduce the following condition:
\begin{description}
\item[(I)] For every $0 < c < d$, there is $\delta > 0$ such that $\vp_{k}(t) < t - \delta$ for all $t \in [c, d]$ and $k \in \mathbb{N}$.
\end{description}
As a consequence, since the hypotheses involved in the statement of Theorem~\ref{exist1} are fulfilled, we can affirm that the following problem}
\begin{eqnarray}
x_{k}(t) & = &  x_k(t_0)  - h_{k}(t_{0}, x_{k}(t_0)) + h_{k}(t, x_{k}(t)) + \displaystyle \int_{t_{0}}^{t} f_{k}(s,x_{k}(s)) {\rm d} g(s), \; \; t\in J, \label{hmde5} \vspace{2mm}\\
x_k (t_0) & = &  \widehat{x}_{k} \label{hmde6}
\end{eqnarray}
 has a solution $x_{k}(\cdot)$ for $k \in \mathbb{N}$.

We are in a position to present the following result of continuous dependence on parameters.

\begin{theorem} \label{thm4.1} Under the above conditions,   assume further that condition  (I) holds and that  the
following conditions are fulfilled:
\begin{itemize}
\item[(i)] The sequence $\widehat{x}_{k} \to x_{0}$ as $k \to \infty$.
\item[(ii)] The sequence  $h_{k}(t, u) \to h(t, u)$ as $k \to \infty$ uniformly for $t \in J$ and $u$ in bounded subsets of $\rn$.
\item[(iii)] The sequence $f_{k}(t, u) \to f(t, u)$ as $k \to \infty$  for $t \in J$ and uniformly for $u$ in bounded subsets of $\rn$, {the sequence of  functions $(M_k)_k$ converges pointwise on $J$ to a function $M \colon J\to\rn$},
and there exists a constant $C > 0$ having the following property: for every subdivision $d = \{\sigma_{0}, \ldots, \sigma_{l} \}$
of the interval $J$ and every finite subset $\{m_{1}, \ldots, m_{l} \}$ of $\mathbb{N}$ the inequality
$\sum_{j = 1}^{l} \int_{\sigma_{j-1}}^{\sigma_{j}} M_{m_{j}}(s) {\rm d} g \leqslant  C$ holds.
\end{itemize}
If there exists $R > 0$ such that
\begin{equation}
\liminf_{k \to \infty} \sup_{r \geqslant R} \frac{\vp_{k}(r)}{r}  <  1, \label{equ5.1}
\end{equation}
then there is a subsequence of $(x_{k})_{k}$ that converges uniformly on $J$ to a solution  $x(\cdot)$ of problem
HMDE \eqref{hmde1}--\eqref{hmde2}.
\end{theorem}
\begin{proof} We separate the proof in four steps.

{\bf Step 1.} We show that the set of functions $\{x_{k}: k \in \mathbb{N} \}$ is bounded on $J$. First, we observe that
$h_{k}(t_{0}, \widehat{x}_{k}) \to h(t_{0}, x_{0})$ as $k \to \infty$. In fact,
\begin{eqnarray*}
\|h_{k}(t_{0}, \widehat{x}_{k}) -  h(t_{0}, x_{0}) \| & \leqslant & \|h_{k}(t_{0}, \widehat{x}_{k}) -  h_{k}(t_{0}, x_{0}) \| +
\|h_{k}(t_{0}, x_{0}) -  h(t_{0}, x_{0}) \| \\
& \leqslant & \vp_{k}(\|\widehat{x}_{k}  -   x_{0} \|) + \|h_{k}(t_{0}, x_{0}) -  h(t_{0}, x_{0}) \| \\
& < & \|\widehat{x}_{k}  -   x_{0} \| + \|h_{k}(t_{0}, x_{0}) -  h(t_{0}, x_{0}) \| \to 0, \; k \to \infty.
\end{eqnarray*}
Moreover, from \eqref{hmde5}--\eqref{hmde6}, we have
\begin{eqnarray*}
\|x_{k}(t)\| & \leqslant & \|\widehat{x}_{k}  - h_{k}(t_{0}, \widehat{x}_{k})\| + \|h_{k}(t, x_{k}(t)) - h_{k}(t, 0) \|+ \|h_{k}(t, 0) \| +
\int_{t_{0}}^{t_{0} +a} M_{k}(s) {\rm d} g(s) \\
& \leqslant &   \|\widehat{x}_{k}  - h(t_{0}, x_{0})\|  +   \|h(t_{0}, x_{0}) - h_{k}(t_{0}, \widehat{x}_{k})\| + \| h_{k}(t, 0) \|+
\vp_{k}(\|x_{k}(t)\|) + \int_{t_{0}}^{t_{0} +a} M_{k}(s) {\rm d} g(s) \\
& \leqslant & C_{1} + \vp_{k}(\|x_{k}(t)\|)
\end{eqnarray*}
for certain constant $C_{1} \geqslant 0$ and all $t \in J$. Let ${\displaystyle N_{k} = \sup_{t \in J} \|x_{k}(t)\|}$. It follows from the
above estimates that
\begin{equation}
1 \leqslant \frac{C_{1}}{N_{k}} + \frac{\vp_{k}(\|x_{k}(t)\|)}{N_{k}} \leqslant \frac{C_{1}}{N_{k}} + \frac{\vp_{k}(N_{k})}{N_{k}} \label{equ5.2}
\end{equation}
for all $k \in \mathbb{N}$. Assume that the set $\{N_{k} : k \in \mathbb{N} \}$ is unbounded. It follows from the estimates \eqref{equ5.2}
that ${\displaystyle \liminf_{k \to \infty} \frac{\vp_{k}(N_{k})}{N_{k}} \geqslant 1}$.
On the other hand, it follows from \eqref{equ5.1} that
\[
\liminf_{k \to \infty} \frac{\vp_{k}(N_{k})}{N_{k}} \leqslant \liminf_{k \to \infty} \sup_{r \geqslant R} \frac{\vp_{k}(r)}{r} < 1
\]
which contradicts our previous assertion. Hence,  we get that there is a constant $N > 0$ independent of $k$ such that
$\|x_{k}(t)\| \leqslant N$ for all $t \in J$.

{\bf Step 2.}  For each $v \in \rn$ the function $h(\cdot, v)$ is  regulated because it is the uniform limit
of the sequence of regulated functions $h_{k}(\cdot, v)  $ as $k \to \infty$. For the same reason, the set
$\{h_{k}(\cdot, v):  k \in \mathbb{N} \}$ is equiregulated on $J$. In addition,
for a bounded set $K \subset \rn$,  the set of functions $\{h_{k}(\cdot, v): v \in K, \; k \in \mathbb{N} \}$ is equiregulated on $J$.
In order to prove that the left and right limits of the functions $h_{k}(\cdot, v)$ exist uniformly for
$v \in K$ and $k \in \mathbb{N}$, it suffices to show that we can approximate the functions $h_{k}(\cdot, v)$ by functions taken
from an equiregulated set of functions. Specifically, for every $\ve > 0$,  since $K$ is a relatively compact set in $\rn$,  there
exist $v_{i} \in \rn$ for $i = 1, \ldots, m$ such that
for each $v \in K$ there is $i = 1, \ldots, m$ with $\|v - v_{i}\| \leqslant \ve$. By our previous remark, the set
$\{h_{k}(\cdot, v_{i}):  k \in \mathbb{N}, \; i =1, \ldots, m \}$ is equiregulated on $J$. Moreover, for $v \in K$, by choosing
$i = 1, \ldots, m$ such that $\|v - v_{i}\| \leqslant \ve$,  we have
\[
\|h_{k}(t, v) - h_{k}(t, v_{i}) \| \leqslant \vp_{k}(\|v - v_{i}\|) \leqslant \ve
\]
for all $  k \in \mathbb{N}$ and $t \in J$. This shows that $\{h_{k}(\cdot, v): v \in K, \; k \in \mathbb{N} \}$ is equiregulated on $J$.

{\bf Step 3.} The set of functions $\{x_{k}(\cdot):  k \in \mathbb{N} \}$ is equiregulated on $J$.
It follows from \eqref{hmde5} that
\[
x_{k}(t) = A_{k} x_{k}(t) + B_{k} x_{k}(t).
\]
First we prove that the set $\{B_{k} x_{k} : k \in \mathbb{N} \}$ is equiregulated on $J$. Let
\[
p_{k}(t) = \int_{t_{0}}^{t} f_{k}(s,x_{k}(s)) {\rm d} g(s), \; t \in J.
\]
It follows from (iii) that the set $\{p_{k}(t) : k \in \mathbb{N}, t \in J \}$ is bounded. Moreover,
for $t_{0} \leqslant t < t_{0} + a$, from \eqref{B} and applying Theorem~\ref{thm1.2}, we have that
{\[
B_{k} x_{k}(t^{+}) - B_{k} x_{k}(t) = f_{k}(t,x_k(t)) \Delta^{+}g(t)
\]}
which shows that $\{B_{k} x_{k} : k \in \mathbb{N} \}$ is right equiregulated at $t$. In similar way, we can show that
$\{B_{k} x_{k} : k \in \mathbb{N} \}$ is left equiregulated at $t \in (t_{0}, t_{0} + a]$.

On the other hand, for $t_{1}, t_{2} \in J$,  we can write
\[
x_{k}(t_{2}) - x_{k}(t_{1})= A_{k} x_{k}(t_{2}) - A_{k} x_{k}(t_{1}) + B_{k} x_{k}(t_{2}) - B_{k} x_{k}(t_{1})
\]
from which we have
\begin{eqnarray}
\|x_{k}(t_{2}) - x_{k}(t_{1})\| & \leqslant &  \|A_{k} x_{k}(t_{2}) - A_{k} x_{k}(t_{1})\| + \|B_{k} x_{k}(t_{2}) - B_{k} x_{k}(t_{1})\|
\nonumber \\
& \leqslant & \vp_{k}(\|x_{k}(t_{2}) - x_{k}(t_{1})\| + \|B_{k} x_{k}(t_{2}) - B_{k} x_{k}(t_{1})\|. \label{equ4.1}
\end{eqnarray}
Assume  that $\{x_{k}(\cdot):  k \in \mathbb{N} \}$ is not equiregulated. Specifically, to fix ideas, suppose that condition
(i) from Definition~\ref{D3}  is not verified at point $\tau_{0}$. Then we can affirm that there exists $\ve > 0$  and two sequences
$(t_{1}^{k})_{k}$ and $(t_{2}^{k})_{k}$ that converge to $\tau_{0}^{-}$ and $\|x_{k}(t_{2}^{k}) - x_{k}(t_{1}^{k})\| \geqslant \ve$ for all
$k \in \mathbb{N}$. Using \eqref{equ4.1} and condition (I) with $\ve$ instead of $c$ and $2 N$ instead $d$, we can choose $\delta > 0$
such that
\[
\|x_{k}(t_{2}^{k}) - x_{k}(t_{1}^{k})\|  \leqslant \|x_{k}(t_{2}^{k}) - x_{k}(t_{1}^{k})\| - \delta + \|B_{k} x_{k}(t_{2}^{k}) -
B_{k} x_{k}(t_{1}^{k})\|
\]
and taking the limit above as $k \to \infty$, we derive that $\delta = 0$, which is a contradiction.

{\bf Step 4.}  From Step 1, Step 3, and Lemma~\ref{L1}, we infer that the set
$\{x_{k}(\cdot):  k \in \mathbb{N} \}$ is relatively compact in the space $G(J, \rn)$.
Consequently,  $(x_{k})_{k \in \mathbb{N}}$ contains a subsequence which is uniformly convergent on $J$.
Without loss of generality, we can denote this subsequence again by $(x_{k})_{k\in\mathbb{N}}$.
Let $x = \lim_{k \to \infty} x_{k}$. In particular,
\[
x(t_{0}) = \lim_{k \to \infty} x_{k}(t_{0}) = \lim_{k \to \infty}\widehat {x}_{k}={x}_{0}.
\]
Moreover, we point out that $h_{k}(t, x_{k}(t)) \to h(t, x(t))$ as $k \to \infty$ uniformly on $J$. In fact,
\begin{eqnarray*}
\|h_{k}(t, x_{k}(t)) - h(t, x(t)) \| & \leqslant & \|h_{k}(t, x_{k}(t)) - h_{k}(t, x(t)) \| + \|h_{k}(t, x(t)) - h(t, x(t)) \|  \\
& < & \| x_{k}(t) -  x(t) \| + \|h_{k}(t, x(t)) - h(t, x(t)) \|
\end{eqnarray*}
and the assertion is consequence of {hypothesis} (ii).
In addition, ${\int_{t_{0}}^{t} f_{k}(s, x_{k}(s)) {\rm d} g \to \int_{t_{0}}^{t} f(s, x(s)) {\rm d} g}$ as $k \to \infty$.
Indeed, in the first place, we can see that
$f_{k}(s, x_{k}(s)) \to f(s, x(s))$ as $k \to \infty$  for $s \in J$. As a matter of fact,
gathering the hypothesis (iii) with the continuity of the function $ f(t, \cdot)$, we obtain that
\[
\|f_{k}(t, x_{k}(t)) - f(t, x(t)) \|  \leqslant  \|f_{k}(t, x_{k}(t)) - f(t, x_{k}(t)) \| + \|f(t, x_{k}(t)) - f(t, x(t)) \|
\to 0, \; k \to \infty.
\]
Furthermore, let  $d = \{\sigma_{0}, \ldots, \sigma_{l} \}$ be a subdivision
of the interval $[t_{0}, t_{0} + a]$ and let $\{m_{1}, \ldots, m_{l} \}$ be a finite subset of $\mathbb{N}$. Then
applying our hypothesis (iii), we have
\[
\left\| \sum_{j = 1}^{l} \int_{\sigma_{j -1}}^{\sigma_{j}} f_{m_{j}}(s, x_{m_{j}}(s)) {\rm d} g(s) \right\| \leqslant \sum_{j = 1}^{l}
\int_{\sigma_{j -1}}^{\sigma_{j}} M_{m_{j}}(s) {\rm d} g(s) < C.
\]
Thus, from Theorem~\ref{DCT}, we obtain the assertion.

Finally, collecting the properties established above, taking limit in \eqref{hmde5} as $k \to \infty$, we obtain
\[
x(t) = x_{0}   - h(t_{0}, x_{0}) + h(t, x(t)) + \int_{t_{0}}^{t} f(s, x(s)) {\rm d} g(s), \; \; t\in J,
\]
which shows that $x(\cdot)$ is a solution of problem HMDE \eqref{hmde1}--\eqref{hmde2}.
\end{proof}

\bigskip
\noindent {\bf Acknowledgements}

We dedicate this work to the memory of our dear co-author Hern\'an R. Henr\'iquez who passed away  on Thursday, June 2, 2022.

C. A. Gallegos is supported by ANID/FONDECYT postdoctorado No 3220147 ; H. R. Henr\'{\i}quez  was partially supported
by Vicerrector\'{\i}a de Investigaci\'on, Desarrollo e Innovaci\'on from Universidad de Santiago de Chile under Grant  DICYT-USACH 041733HM; J. G. Mesquita was partially supported by CNPq grant 307582/2018-3.

\end{document}